\documentclass[12pt]{article}
\usepackage{fancyhdr}
\usepackage{e-jc}
\usepackage{amsthm,amsmath,amssymb}
\usepackage{graphicx}
\def\C{\mathbb C}
\def\R{\mathbb R}

\def\epsilon{\varepsilon} 
\def\eps{\varepsilon} 
\newtheorem{thm}{Theorem}[section]
\newtheorem{fact}{Fact}
\newtheorem{claim}[thm]{Claim}

\newtheorem{lem}[thm]{Lemma}
\newtheorem*{conjecture}{Conjecture}
\theoremstyle{definition}

\theoremstyle{remark}

\numberwithin{equation}{section}
\newenvironment{verification}[1][\proofname]{%
  \proof[Verification.]%
}{\endproof}

\title{\bf Dimensions of the irreducible representations of the symmetric and alternating group}
\date{}

\author{ Korneel Debaene\\
}

\begin{document}

\maketitle

\thispagestyle{fancy}
\pagestyle{fancy}

\begin{abstract} We establish the existence of an irreducible representation of $A_n$ whose dimension does not occur as the dimension of an irreducible representation of $S_n$, and vice versa. This proves a conjecture by Tong-Viet. The main ingredient in the proof is a result on large primefactors in short intervals.

 \end{abstract}

\section{Introduction and statement of results}

To what extent are groups determined by their characters? On the one hand, there are examples known of non-isomorphic groups with exactly the same character table, e.g. $Q_8$ and $D_8$. On the other hand, for nonabelian simple groups, a conjecture by Huppert predicts quite a different behaviour.
 
\begin{conjecture} Let $G$ be a finite group and $H$ be a finite nonabelian simple group such that the sets of character degrees of $G$ and $H$ are the same. Then there exists an abelian group $A$ such that $G\cong H\times A$.
\end{conjecture}
This conjecture has been verified for many simple groups, but remains open for the alternating groups $A_n$ when $n\geq14$. 

In this regard, Tong-Viet proved that the {\em multiset} of dimensions of irreducible representations of $A_n$ determines $A_n$ in \cite{TongAlt}. He conjectured that the set of dimensions of irreducible representations is different for $A_n$ and $S_n$ in \cite{Tong}, and proved it for the special case that $n$ is of the form $2p+1$. 

We will prove the following theorem, which proves Tong-Viets conjecture, and gives some indication in favour of Huppert's Conjecture.
\begin{thm}
\label{Theorem1} The set $\{\dim \rho\mid \rho \mbox{ irreducible representation of }S_n\}$ is not equal to the set $\{\dim \rho\mid \rho \mbox{ irreducible representation of }A_n\}$, for all $n\geq 3$.
\end{thm}
In fact, the proof is constructive and gives a specific irreducible representation of $S_n$ whose dimension does not occur as a dimension of any irreducible representation of $A_n$, and a specific irreducible representation of $A_n$ whose dimension does not occur as a dimension of any irreducible representation of $S_n$. 

The proof depends on the size of $n$ in the following way. For large $n$ we give a proof using a lemma on the existence of prime numbers in relatively short intervals. For medium sized $n$ we verify this lemma by computer, and for very large $n$ this is implied by a theorem by Schoenfeld \cite{Schoenfeld}. For small $n$ on the other hand, we need to verify the theorem directly, with the help of a computer.

The proof uses the well known description of the irreducible representations of $S_n$ and $A_n$ in terms of Young diagrams of partitions. We introduce some notation. A partition of a natural number $n$ is a non-increasing sequence $\lambda=(\lambda_1,\lambda_2, \cdots, \lambda_s)$ of natural numbers with sum $\sum \lambda_i = n$. The corresponding Young diagram consists of boxes for each $i,j\geq 1$ such that $j\leq \lambda_i$. We use the matrix notation and refer to the box on the $i$-th row and $j$-th column by $(i,j)$. The conjugate partition $\lambda^{*}$ is the partition corresponding to the transpose of the diagram corresponding to $\lambda$. Given a box at position $(i,j)$, its arm is the set of boxes at positions $\{(i,k)\mid j<k\}$ and its leg is the set of boxes at positions $\{(k,j)\mid i<k\}$. The hook length $h(\lambda)_{(i,j)}$ of a box is the sum of the cardinalities of its arms and legs plus one. The hook product of a diagram is the product of the hook lengths of all boxes. The hook product of the diagram corresponding to the partition $\lambda$ will be denoted by $\Pi(\lambda)$.

The irreducible representations of $S_n$ correspond one-to-one to partitions $\lambda$ of $n$. For each partition $\lambda$, the dimension of the irreducible representation $\rho_\lambda$ can be recovered from the hook product $\Pi(\lambda)$ of the Young diagram corresponding to $\lambda$ through the formula $\dim \rho_\lambda = \frac{n!}{\Pi(\lambda)}$.

For $A_n$, the correspondence is a bit more subtle. Each pair $(\lambda, \lambda^*)$ of conjugate partitions of $n$, where $\lambda\neq\lambda^*$, corresponds to an irreducible representation $\rho_\lambda$ of $A_n$, where the formula for the dimension is the same as above; $\dim \rho_\lambda = \frac{n!}{\Pi(\lambda)}$. Each self-conjugate partition $\lambda$ corresponds to a pair of irreducible representations $(\rho_{\lambda,1}, \rho_{\lambda,2})$, both having dimension $\dim \rho_{\lambda,i} = \frac{n!}{2\Pi(\lambda)}$.
For details we refer the reader to \cite[Chapter 4 and 5]{FultonHarris}
\section{Proofs}

We will use two facts about prime numbers in short intervals. The first fact holds only for $k\geq 337$, and this is the reason why we have to restrict to this range in the main Theorem of this section.
\begin{lem} \label{Primelemma1} For all integers $k\geq337$ the interval $[k-\lfloor\frac{k}{20}\rfloor,k]$ contains at least two prime numbers.
\end{lem}
\begin{proof}
A result by Schoenfeld \cite{Schoenfeld} states that there is always a prime in $(x,x+\frac{x}{16 597}]$ for all $x\geq 2 010 760$, which proves the lemma for $k\geq 2 010 760$. We have checked the result for the remaining values of $k$ with the help of a computer.
\end{proof}
This type of short interval suffices for our purposes, but it should be noted that much stronger results exist, at least asymptotically. It is known \cite{BHP} that there exists an integer $k_0$ such that for all $k>k_0$ there is a prime in the interval $[k-k^{21/40}, k]$. However, the result is ineffective - it gives no upper bound on $k_0$.

The following is a special case of Lemma 3.3 from \cite{BBOO}.
\begin{lem}\label{PrimelemmaBBOO}If $1\leq h\leq y< k$ are integers for which \[\left(\frac{k}{h}\right)^{h}\geq(k+h)^{\pi(y)},\] then one of the integers $k+1,...,k+h$ has a prime factor bigger than $y$.
\end{lem}
We apply this lemma to guarantee the existence of large prime factors in the type of short intervals we need.
\begin{lem} \label{Primelemma2}
For all integers $k\geq337$ and $\frac{1}{2}\sqrt{k}\leq h \leq \frac{3k}{20}$, there exists a prime $p\geq 3h$ such that \[p \mid \prod_{i=1}^h (k+i).\]
\end{lem}
\begin{proof} If $k\geq 2010760$, and $h\leq\frac{3k}{1000}$, we use Lemma~\ref{PrimelemmaBBOO}, with $y=3h$. We plug in the inequality $\pi(y)\leq 1.25506\frac{y}{\log(y)}$ proved by Schoenfeld in \cite{Schoenfeld}. It suffices to show that \[\left(\frac{k}{h}\right)^{h}>\left(\frac{1003}{1000}k\right)^{3.77\frac{h}{\log(3h)}}.\]

Upon taking logarithms, we need to show that \[\log(3h)\log(k/h)>3.77\log(k)+3.77\log(\frac{1003}{1000}).\]

The left hand side is a quadratic function in $\log(h)$ with negative leading coefficient, thus in any interval the minimum value is in one of the endpoints. Thus we check that the equality holds for $h=\frac{1}{2}\sqrt{k}$ and for $h=\frac{3k}{1000}$. Putting $h=\frac{1}{2}\sqrt{k}$, we need to show that \[(\log(3/2)+\frac{1}{2}\log(k))(\log(2)+\frac{1}{2}\log(k))>3.77\log(k)+3.77\log(\frac{1003}{1000}).\]
This inequality holds whenever $\log(2)+\frac{1}{2}\log(k)>7.54$, which holds from $k\geq885322$.
Setting $h=\frac{3k}{1000}$, we need to show that \[\log(\frac{9}{1000}k)\log(\frac{1000}{3})>3.77\log(k)+3.77\log(\frac{1003}{1000}),\]
which is valid for $k\geq676712 $. This proves the lemma for all $k\geq 2010760$, and $h\leq\frac{3k}{1000}$. 

If $k\geq 2010760$ and $\frac{3k}{1000}\leq h\leq\frac{3k}{20}$, we have the stronger result by Schoenfeld \cite{Schoenfeld} which tells us that there is a prime in $(x,x+\frac{x}{16597}]$ for $x\geq 2010760$, which implies us that there is a prime $p$ in $[k+1,k+h]$, which is automatically bigger than $3h$. This proves the lemma for all $k\geq 2010760$.

In the remaining range of $337\leq k\leq 2010760 $ we have checked the statement of the lemma with the help of a computer. In fact, a much stronger statement holds in this range of $k$. For each $337\leq k\leq 2010760 $, with $51$ exceptions, there is a prime in the interval $[k+1, k+\lfloor\sqrt{k}/2\rfloor]$. We run through the prime pairs $(p_n,p_{n+1})$ and check that $p_n+\sqrt{p_n}/2\geq p_{n+1}$, so that for all $k\in [p_n,p_{n+1}), p_{n+1}\in[k+1,k+\lfloor\sqrt{k}/2\rfloor]$. For the exceptional values of $k$ we check the lemma directly.
\end{proof}
We need to include the possibility of $h$ being as big as a fraction of $k$ because the primes we get from Lemma~\ref{Primelemma1} could be smaller than $k$ by a linear term in $k$. Much stronger results are available in shorter intervals, for example Ramachandra proved \cite{Rama} that for a certain $\alpha<1/2$, there is an integer in the interval $[x, x+x^\alpha]$ with a prime factor as big as $x^{15/26}$. The strongest result to date of this kind is \cite{Sharp}, namely the existence of a prime factor of magnitude $x^{0.738...}$ in the interval $[x,x+x^{1/2}]$.

We will repeatedly make use of the following two lemmata.
\begin{lem}\label{Factoriallemma}For all nonnegative integers $x,y$ such that $x-y$ is nonnegative, we have \[(x+y)!(x-y)! \leq x!^2 \frac{e^2}{2\pi}\min(e^{\frac{x+y}{x-y}\frac{y^2}{x}},e^{2\frac{y^2}{x}}).\]
\end{lem}
\begin{proof}
We use the inequalities $\sqrt{2\pi}n^{n+1/2}e^{-n} \leq n!\leq en^{n+1/2}e^{-n},$ combined with the following computations
\begin{align*}
\frac{(x+y)^{x+y}(x-y)^{x-y}}{x^{2x}} &= \frac{(x^2-y^2)^x}{x^{2x}}\left(\frac{x+y}{x-y}\right)^y\\
 &= \left(1-\frac{y^2}{x^2}\right)^x \left(1+\frac{2y}{x-y}\right)^y\\
 &\leq e^{-\frac{y^2}{x}+\frac{2y^2}{x-y}}=e^{\frac{x+y}{x-y}\frac{y^2}{x}},
\end{align*}
\begin{align*}
\frac{(x+y)^{x+y}(x-y)^{x-y}}{x^{2x}} &= (1+\frac{y}{x})^{x+y}(1-\frac{y}{x})^{x-y}\\
 &\leq e^{\frac{y}{x}(x+y)-\frac{y}{x}(x-y)}=e^{2\frac{y^2}{x}}.\qedhere
\end{align*}
\end{proof}

\begin{lem}\label{factorialdivides} Let $\lambda$ be a partition of $n$. Let $h$ be the hook length of a box not on position $(1,1)$. If $2h-n\geq 0$, then $(2h-n)!$ divides the hook product $\Pi(\lambda)$.
\end{lem}
\begin{proof} Since the hook length increases when decreasing either of the coordinates of a box, and since conjugate partitions have the same hook product, we may assume that the box $(1,2)$ has hook length $h'\geq h$. Let $t=\lambda_1-\lambda_2$ and let $a$ be the leg length of the box $(1,2)$. Then $h'=a+\lambda_1-1$. The total number of boxes is at most $n$, so \begin{align*}
&2(a-1)+\lambda_1+\lambda_2\leq n\\
\Longleftrightarrow\quad &2h'-\lambda_1+\lambda_2\leq n\\
\Longleftrightarrow\quad &t\geq 2h'-n\geq 2h-n.
\end{align*}

It now suffices to note that the hook lengths of the $t$ boxes in positions $\{(1,k)\mid\lambda_2<k\leq \lambda_1\}$ are $t,t-1,\cdots,1$.\end{proof}

The following three lemmata will be useful in bounding hook products of partitions.
\begin{lem}\label{Productlemma}Let $N$ be a natural number. For any tuple $(c_1,c_2,\cdots ,c_k)$ of positive integers, we have \[\prod_{j=1}^k(N-j+c_j)\leq\frac{N}{N-\sum_{j=1}^k c_j}\prod_{j=1}^k(N-j).\]
\end{lem}
\begin{proof}
When $(c_1,c_2,\cdots ,c_k)=(1,1,\cdots,1)$, we have equality for every $k$. For the induction step, assume that the statement holds for the tuple $(c_1,c_2,\cdots, c_k)$. We will show that the statement then also holds for the tuple $(c_1,c_2,\cdots,c_i+c_k,\cdots ,c_{k-1})$. Note that any tuple of positive integers can be reached using this operation multiple times, starting from a tuple $(1,1,\cdots,1)$.

In order to show the induction step, it suffices to check that \[\frac{N-i+c_i+c_k}{(N-k+c_k)(N-i+c_i)}\leq \frac{1}{N-k}.\]

Indeed, this reduces to \begin{align*}
(N-k)(N-i+c_i+c_k)&\leq(N-k+c_k)(N-i+c_i)\\
\Longleftrightarrow c_k(N-k)&\leq c_k(N-i+c_i)\\
\Longleftrightarrow i &\leq k+c_i. 
\qedhere
\end{align*}
\end{proof}

\begin{lem}\label{TwoPartitions1}Let $\lambda=(A+1,1^B)$ be a hook partition. Let $\tau$ be a partition of $t$ with first part $\tau_1\leq A$ and with number of parts $\tau_1^{*}\leq B$. Let $\mu$ be the partition obtained by inserting the Young diagram of $\tau$ into positions $\{(i,j)\mid 2\leq i \leq \tau_1+1, 2\leq j\leq\tau_1^{*}+1\}$ of the Young diagram of $\lambda$. Then
\[\Pi(\mu)\leq \Pi(\lambda)\Pi(\tau) \frac{A+1}{A+1-t}\frac{B+1}{B+1-t}.\]
\end{lem}
\begin{proof}
The hook product of $\mu$ equals \[\Pi(\mu)=\Pi(\tau)h(\lambda)_{(1,1)}\prod_{j=1}^{\tau_1}(h(\lambda)_{(1,j+1)}+\tau_j^{*})\prod_{j=\tau_1+1}^{A}h(\lambda)_{(1,j+1)}
\prod_{i=1}^{\tau_1^{*}}(h(\lambda)_{(i+1,1)}+\tau_i)\prod_{i=\tau_1^{*}+1}^{B}h(\lambda)_{(i+1,1)}.\]

Now note that $h(\lambda)_{(1,j+1)}=A+1-j$ and $h(\lambda)_{(i+1,1)}=B+1-i$. Using Lemma \ref{Productlemma} with $N=A+1$ and with the conjugate partition $\tau^{*}$, we get \[\prod_{j=1}^{\tau_1}(h(\lambda)_{(1,j+1)}+\tau_j^{*})\prod_{j=\tau_1+1}^{A}h(\lambda)_{(1,j+1)}\leq \frac{A+1}{A+1-t}A!.\]

Similarly using Lemma \ref{Productlemma} with $N=B+1$ and with the partition $\tau$, we get \[\prod_{i=1}^{\tau_1^{*}}(h(\lambda)_{(i+1,1)}+\tau_i)\prod_{i=\tau_1^{*}+1}^{B}h(\lambda)_{(i+1,1)}\leq \frac{B+1}{B+1-t}B!.\]

Thus \[\Pi(\mu)\leq \Pi(\lambda)\Pi(\tau) \frac{A+1}{A+1-t}\frac{B+1}{B+1-t},\]
since $\Pi(\lambda)=(A+B+1)A!B!$.
\end{proof}

\begin{lem}\label{TwoPartitions2}Let $\lambda=(A+1,B)$ be a partition. Let $\tau=(\tau_1,\cdots,\tau_s)$ be a partition of $t$ with first part $\tau_1\leq B$. Let $\mu$ be the partition $(A+1,B,\tau_1,\cdots,\tau_s)$. Then
\[\Pi(\mu)\leq \Pi(\lambda)\Pi(\tau) \frac{A+3}{A+3-t}\frac{B+1}{B+1-t}.\]
\end{lem}
\begin{proof}
The hook product of $\mu$ equals \[\Pi(\mu)=\Pi(\tau)\prod_{j=1}^{\tau_1}(h(\lambda)_{(1,j)}+\tau_j^{*})\prod_{j=\tau_1+1}^{A+1}h(\lambda)_{(1,j)}
\prod_{i=1}^{\tau_1}(h(\lambda)_{(2,i)}+\tau_i^{*})\prod_{i=\tau_1+1}^{B}h(\lambda)_{(2,i)}.\]

Now note that $h(\lambda)_{(1,j)}=A+3-j$ for $j\leq B$ and $h(\lambda)_{(1,j)}=A+2-j$ for $j> B\geq \tau_1$. As before we use Lemma \ref{Productlemma} with $N=A+3$ and with the conjugate partition $\tau^{*}$. We get 
\[\prod_{j=1}^{\tau_1}(h(\lambda)_{(1,j)}+\tau_j^{*})\prod_{j=\tau_1+1}^{A+1}h(\lambda)_{(1,j)}
\leq\frac{A+3}{A+3-t}\frac{(A+2)!}{A+2-B}.\]

Now note that $h(\lambda)_{(2,i)}=B+1-i$. Again we use Lemma \ref{Productlemma}, with $N=B+1$ and the same partition $\tau^{*}$. We get
\[\prod_{i=1}^{\tau_1}(h(\lambda)_{(2,i)}+\tau_i^{*})\prod_{i=\tau_1+1}^{B}h(\lambda)_{(2,i)} \leq \frac{B+1}{B+1-t}B!.\]

Thus \[\Pi(\mu)\leq \Pi(\lambda)\Pi(\tau) \frac{A+3}{A+3-t}\frac{B+1}{B+1-t},\]
since $\Pi(\lambda)=\frac{(A+2)!B!}{A+2-B}$.
\end{proof}

The following lemma contains a computation that will be used several times in the proof of Theorem \ref{Theorem1.1}. Neither the assumption nor the conclusion is sharp, but it suffices for our purposes.
\begin{lem}\label{ComputationCase2and3}Let $k,r,\eps$ be three integers satisfying \[k\geq 337,\quad r\leq \frac{3k}{20}+1,\quad 0\leq \eps\leq 2r+1,\]
and let $\eta$ be zero or one.
Assume that \[X=\frac{2e^2}{\pi} e^{\frac{k+r}{k-r} \frac{r^2}{k}} \eps! \left(\frac{1}{k+r-\eps}\right)^{\eps-\eta}\geq \frac{1}{2}.\]
Then we can conclude that either \[r-2\eps \geq \sqrt{k}/2,\textup{ and }\,5\eps\leq r-2,\]
or we have that $\eps\in\{0,\eta\}$.
\end{lem}
\begin{proof}
First, note that \[\frac{k+r}{k-r}\leq \frac{\frac{23}{20}k+1}{\frac{17}{20}k-1}\leq \frac{\frac{23}{20}337+1}{\frac{17}{20}337-1}=1.361... \leq 7/5.\]

We assume that $\eps\notin\{0,\eta\}$ and we will show that $\eps$ is quite small, and $r$ quite large.
We use that $\eps!\leq e\sqrt{\eps}(\eps/e)^\eps$ for $\eps\geq 1$, and that \begin{equation}\label{eqn:eq1} e\eps \leq 6e\frac{k}{20} +3e \leq \frac{17}{20}k-2\leq k-r-1\leq k+r-\eps,\quad (k\geq337) 
\end{equation}
to bound \begin{equation}\label{eqn:eq2} X\leq \frac{2e^3}{\pi}e^{\frac{7r^2}{5k}}\sqrt{\eps}\left(\frac{1}{e^2}\right)^\eps(k+r-\eps)^\eta.
\end{equation}

We can now prove that $\eps\leq \frac{7}{10}\frac{r^2}{k} + k^{1/3}$. Assume for the sake of contradiction that $\eps$ were bigger. Then we could bound \[X\leq \frac{2e^3}{\pi}\frac{\sqrt{(6\frac{k}{20}+3)}(k+\frac{3k}{20}+1)^\eta}{e^{2k^{1/3}}}<\frac{1}{2},\quad(k\geq 337)\]
a contradiction. Thus we have that \begin{equation}\label{eqn:eq6}\eps\leq \frac{7}{10}\frac{r^2}{k} + k^{1/3}.
\end{equation} 

Now that we know that $\eps$ is fairly small, we can use this to show that $\eps$ is even smaller. Note that \[\frac{7}{10}\frac{r^2}{k} + k^{1/3} \leq \frac{7}{10}(\frac{3}{20}+\frac{1}{337})(\frac{3k}{20}+1) + \frac{k}{337^{2/3}}\leq 0.108+0.037k.\]

In place of \eqref{eqn:eq1}, we have 
\begin{equation}\label{eqn:eq3} e^3\eps \leq e^3(0.108+0.037k) \leq \frac{17}{20}k-2\leq k-r-1\leq k+r-\eps,\quad (k\geq337)
\end{equation}
so that in place of \eqref{eqn:eq2} we even have \begin{equation}\label{eqn:eq4} X\leq \frac{2e^3}{\pi}e^{\frac{7r^2}{5k}}\sqrt{\eps}\left(\frac{1}{e^4}\right)^\eps(k+r-\eps)^\eta.
\end{equation}

As we concluded our bound \eqref{eqn:eq6} from \eqref{eqn:eq2}, we can now conclude from \eqref{eqn:eq4} that \begin{equation}\label{eqn:eq7}\eps\leq \frac{7}{20}\frac{r^2}{k} + \frac{k^{1/3}}{2}.
\end{equation} 

We now claim that $r\geq \frac{7}{5}\sqrt{k}$. Assume for the sake of contradiction that $r$ were smaller. Then, using the definition of $X$, we could bound \begin{equation}\label{eqn:eq5}X\leq \frac{2e^2}{\pi}e^{(7/5)^3}\frac{\eps!}{(k+r-\eps)^\eps}(k+r-\eps)^\eta.
\end{equation}

Since $\frac{\eps!}{(k+r-\eps)^\eps}$ is a decreasing function of $\eps$ in our range of $\eps$ (e.g. in the range $\eps\leq k/3$), we may bound \eqref{eqn:eq5} by putting $\eps=\eta+1$ and obtaining \[X\leq \frac{2e^2}{\pi}\frac{e^{(7/5)^3}2}{k-\eps}<\frac{1}{2},\quad (k\geq337)\]
a contradiction. Hence $r\geq \frac{7}{5}\sqrt{k}$.
This, together with the bound \eqref{eqn:eq7} is enough to prove that the two desired inequalities hold.
\begin{align*}
r-2\eps \geq r(1-\frac{7r}{10k})-k^{1/3} &\geq r(1-\frac{21}{200}-\frac{7}{10k})-k^{1/3}\\
&\geq 0.89r-k^{1/3}\\
 &\geq 1.246\sqrt{k}-k^{1/3}\geq \frac{\sqrt{k}}{2},\quad (k\geq337)
\end{align*}
\begin{align*}
5\eps \leq \frac{7}{4}\frac{r^2}{k} + \frac{5}{2}k^{1/3} &\leq r(\frac{21}{80}+\frac{7}{4k})+\frac{5}{2}k^{1/3}\\
&\leq 0.227r+\frac{5}{2}k^{1/3} \\
&\leq r - 2 + (\frac{5}{2}k^{1/3} +2-0.733\frac{7}{5}\sqrt{k}) \leq r - 2 +0.56.\quad (k\geq337)
\end{align*}

Since $\eps$ and $r$ are integers, this implies $5\eps\leq r-2$.
\end{proof}

We can now give the proof of Theorem~\ref{Theorem1} for big enough $n$.
\begin{thm}\label{Theorem1.1} The set $\{\dim \rho\mid \rho \mbox{ irreducible representation of }S_n\}$ is not equal to the set $\{\dim \rho\mid \rho \mbox{ irreducible representation of }A_n\}$, for all $n \geq 677$
\end{thm}
\begin{proof}
For odd $n=2k+1$, we consider the partition 
\begin{align*}\lambda_k = (k+1 , 1^k),
\end{align*}
whereas for even $n=2k+2$, we consider the partition 
\begin{align*}\lambda^{'}_k = (k+1 , 2, 1^k).
\end{align*} 

Since the proof for odd and even $n$ is completely analogous, we only give the proof for odd $n$. 

The partition $\lambda_k$ has a symmetric Young diagram corresponding to an irreducible representation $\rho_k$ in $S_n$, and a pair of irreducible representations of $A_n$ which we'll denote by $(\rho_{k,1},\rho_{k,2})$. Then $\dim\rho_k=2\dim\rho_{k,i}$. The hook product of $\lambda_k$ is \begin{equation}\label{eq:hookproduct}\Pi_k=(2k+1)k!^2.\end{equation} To prove that there is no other irreducible representation of either $S_n$ or $A_n$ with the same dimension as $\rho_k$ or $\rho_{k,i}$, we show that there does not exist any other partition with a hook product equal to $\Pi_k$, $2\Pi_k$ or $\frac{1}{2}\Pi_k$, so that the result follows by the hook product formula.\\
We first reduce the possible shape of a partition with a such hook product by using Lemma~\ref{Primelemma1} and considering two primes $p,q$ in the interval $[k-\left\lfloor\frac{k}{20}\right\rfloor,k]$. Since $p^2q^2\mid\Pi_k$, the diagram contains two boxes of hook length a multiple of $p$ and two boxes of hook length a multiple of $q$. In the case of $n$ even, we have $\Pi(\lambda_k^{'})=\Pi_k^{'}=(2k+1)(k+1)^2(k-1)!^2$. In order to guarantee that again $p^2q^2\mid\Pi_k^{'}$, we must apply Lemma \ref{Primelemma1} to $k-1$, which is why the Theorem assumes $k\geq 338$. 

A partition of $2k+1$ cannot have a box of hook length $3p$ or higher since $3p>2k+1$. So unless there is a box of hook length $2p$, there are two boxes of hook length $p$. 

The first step is to show that there cannot both be a box of hook length $2p$ and $2q$. Assume for the sake of contradiction that there are boxes of hook length $2p$ and $2q$. Then there is a box not on position $(1,1)$ with hook length at least $2q$. By Lemma \ref{factorialdivides} we have that $(4q-2k-1)!$ divides the hook product of the partition. However, since certainly \[4q-2k-1\geq 2k-4\left\lfloor\frac{k}{20}\right\rfloor -1>\frac{3}{2}k,\] Lemma~\ref{Primelemma1} implies that there is a prime in the interval $[k+1, 4q-2k+1]$ which divides the hook product but does not divide $\Pi_k$, $2\Pi_k$ or $\frac{1}{2}\Pi_k$. 

\begin{figure}[h]
\centering
\includegraphics[width=6cm]{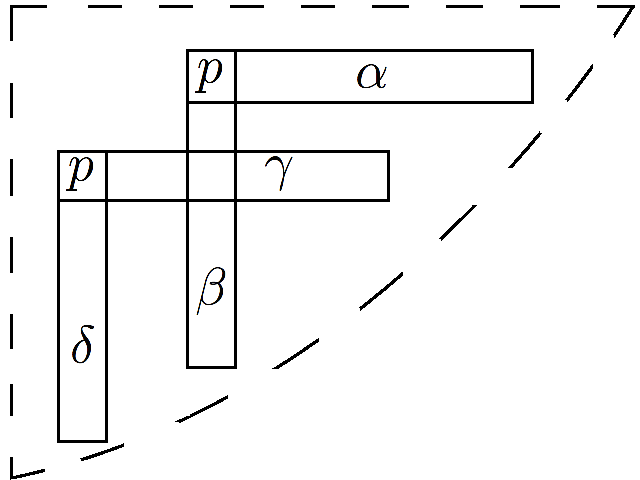}
\caption{Shape of hypothetical partition with hook product $\Pi_k$}
\end{figure}

So, possibly switching $q$ for $p$, we may assume the diagram to have the shape indicated in Figure 1. We denote the arms of the boxes by $\alpha$ and $\gamma$, and the legs of the boxes by $\beta$ and $\delta$ We denote the number of boxes in $\alpha, \beta, \gamma, \delta$ by $a,b,c,d$. Thus \[\begin{cases}a+b = p-1\\c+d = p-1\end{cases}.\] 

Now note that the two boxes of hook length $p$ together with their arms and legs already cover $2p$ or $2p-1$ of the $2k+1$ boxes, so there are at most $2k+1 - (2p-1)\leq 2\left\lfloor\frac{k}{20}\right\rfloor+2$ boxes elsewhere. One important consequence is that if it holds that $c \geq 2\left\lfloor\frac{k}{20}\right\rfloor+2$, we can conclude that $\gamma$ lies on the second row; a row above $\alpha$ or between $\alpha$ and $\gamma$ would imply there to be more than $2p+c>2k+1$ boxes. Similarly, if $b\geq 2\left\lfloor\frac{k}{20}\right\rfloor+2$ then $\beta$ is on the second column. 

We now distinguish three cases of different qualitative behaviour.

\noindent \textbf{ Case one : $a,d \leq k-3\left\lfloor\frac{k}{20}\right\rfloor-3$} 

\begin{figure}[h]
\centering
\includegraphics[width=6cm]{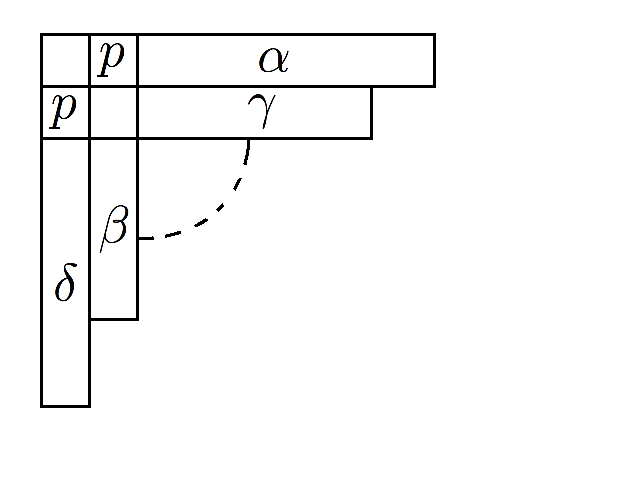}
\caption{Shape of partition in Case 1}
\end{figure}

\mdseries In this case, \[c=p-1-d\geq p-1-k+ 3\left\lfloor\frac{k}{20}\right\rfloor+3\geq  2\left\lfloor\frac{k}{20}\right\rfloor+2,\] which implies that $\gamma$ lies on the second row. Similarly, $\beta$ lies on the second column. This fixes the position of the two boxes of hook length $p$ to be $(1,2)$ and $(2,1)$. Then we can assume that Case 1 does not occur; consider the prime $q$. There cannot be a box of hook length $2q$ since this would be the upper left box and we would have \[2q=a+d+3\leq 2k-6\left\lfloor\frac{k}{20}\right\rfloor-3< 2k-4\left\lfloor\frac{k}{20}\right\rfloor\leq 2q,\] which is impossible. After switching $q$ for $p$, we are no longer in Case 1 since $(1,2)$ and $(2,1)$ do not have hook length $q$. 

\noindent \textbf{Case two : $a,d>k-3\left\lfloor\frac{k}{20}\right\rfloor-3$} 

\begin{figure}[h]
\centering
\includegraphics[width=6cm]{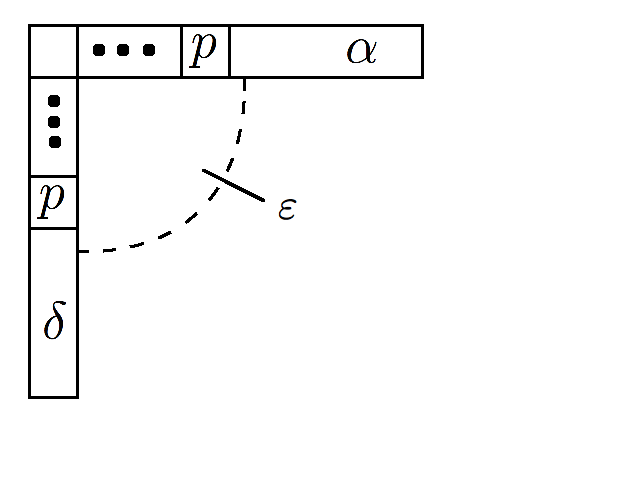}
\caption{Shape of partition in Case 2}
\end{figure}

\mdseries This implies that $\alpha$ is on the first row and $\delta$ is on the first column. Denote by $\epsilon$ the number of boxes not on the first row or column. Let $r$ be such that the first row consists of $k+r-\eps+1$ boxes. It follows that the first column consists of $k-r+1$ boxes. Since $k-r+1\geq 2+d\geq k-3\lfloor\frac{k}{20}\rfloor$, we have that \[r\leq 3\left\lfloor\frac{k}{20}\right\rfloor+1.\]

Since $k+r-\eps+1 \geq 2+a \geq  k-3\lfloor\frac{k}{20}\rfloor$, we also have that \[\epsilon\leq 6\left\lfloor\frac{k}{20}\right\rfloor+2.\] 


We may use Lemma \ref{TwoPartitions1} to bound the hook product of the partition
\begin{align*}\Pi &\leq (2k+1-\eps)(k+r-\epsilon)!(k-r)!\epsilon! \frac{k+r-\eps+1}{k+r-2\eps+1} \frac{k-r+1}{k-r-\eps+1}\\
& \leq 4(2k+1)(k+r)!(k-r)!\frac{1}{(k-r-\eps)^\eps}\epsilon!,
\end{align*}
since the two last factors are easily seen to be bounded by 2 by using the upper bounds on $\eps$ and $r$. We use Lemma~\ref{Factoriallemma} to further bound the hook product

\begin{equation*}\Pi \leq \Pi_k \frac{2e^2}{\pi}e^{\frac{k+r}{k-r}\frac{r^2}{k}}\left(\frac{1}{k+r-\epsilon}\right)^\epsilon \eps!.
\end{equation*}

Since by assumption $\Pi$ is either $\frac{1}{2}\Pi_k$, $\Pi_k$ or $2\Pi_k$, we deduce that \begin{equation*}\frac{1}{2}\leq \frac{2e^2}{\pi}e^{\frac{k+r}{k-r}\frac{r^2}{k}}\left(\frac{1}{k+r-\epsilon}\right)^\epsilon \eps!.
\end{equation*}

This means that our situation satisfies the conditions of Lemma \ref{ComputationCase2and3}, with $\eta=0$, and we may conclude that either $\eps=0$ or $r-2\eps\geq \sqrt{k}/2$, and $5\eps\leq r-2$. We deal with the case $\eps\neq 0$ first.

We now consider the prime factors of the hook product. We see that the hook product contains as factors $(k+r-2\epsilon)!$, and $(k-r-\epsilon)!$. 
Since $r-2\epsilon\leq r-2\leq \frac{3}{20}k$, and 
\[r-2\epsilon \geq \frac{1}{2}\sqrt{k},\] we can consider, according to Lemma~\ref{Primelemma2} a prime $p'>3(r-2\epsilon)$ such that $p' \mid \frac{(k+r-2\epsilon)!}{k!}$. Now since 
$ 5\epsilon \leq r-2$, we deduce that $p'>2r-\epsilon$ and so $p' \nmid \frac{k!}{(k-r-\epsilon)!}$, and arrive at a contradiction; there are more factors of $p'$ in $\Pi$ than there are in $\Pi_k$. We conclude that no partition satisfying the conditions of Case 2 with $\epsilon\neq 0$ has a hook product equal to $\frac{1}{2}\Pi_k, \Pi_k$, or $2\Pi_k$.

Finally, we address the case $\epsilon=0$. The hook product $(2k+1)(k+r)!(k-r)!$ is then strictly bigger than $\Pi_k$, and if $r\leq\frac{1}{2}\sqrt{k}$ were to hold, smaller than \[(2k+1)k!^2\frac{e^2}{2\pi}e^{\frac{7r^2}{5k}}<\frac{e^2}{2\pi}e^{7/20}\Pi_k<2\Pi_k.\]

Thus we have $r-2\epsilon\geq\frac{1}{2}\sqrt{k}$ and we may finish as above.\\

\noindent \textbf{Case three : $a>k-3\left\lfloor\frac{k}{20}\right\rfloor-3\geq d$.}

\begin{figure}[h]
\centering
\includegraphics[width=6cm]{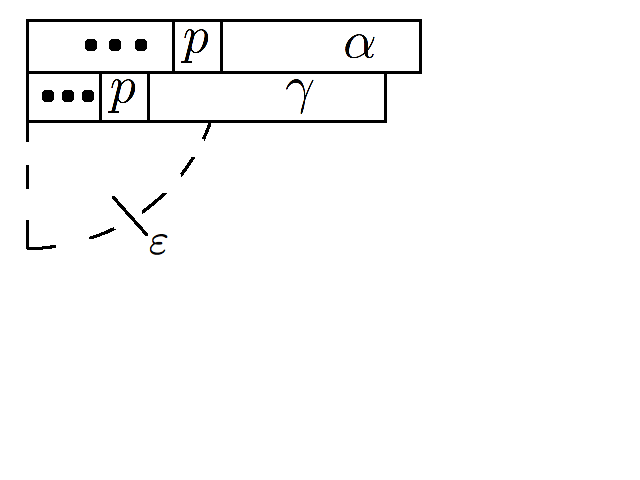}
\caption{Shape of partition in Case 3}
\end{figure}

\mdseries As in Case 1, we have that $c\geq 2\left\lfloor\frac{k}{20}\right\rfloor +2$, and so we immediately deduce that $\gamma$ lies on the second row, so that the boxes of hook length $p$ lie on the first and second row. Now consider the prime $q$. There cannot be a box of hook length $2q$ because its arm cannot cover both $\alpha$ and $\gamma$ and the total number of boxes would be at least \[2q+\min(a,c) \geq 2k+2,\] more than the total number of available boxes. Thus there are two boxes of hook length $q$. We now claim that these also lie on the first and second row. If not, since a box of hook length $q$ not on the first two rows can only lie in one of $\beta$ and $\delta$, the number of boxes would be at least \[q+\min(b,d)+a+c+1= q+p+\min(a,c)\geq 2k+2.\] 

Thus, there are boxes of hook length $p$ and $q$ on both the first and second row. This forces $\delta$ to be very short. Assume that $p$ is the largest of the two primes. Then $\delta$ is disjoint from the two hooks of the boxes of hook length $q$, so $d+2+2q\leq 2k+1$, so 

\[d\leq 2k-1-2q\leq2\left\lfloor\frac{k}{20}\right\rfloor-1.\]

If $q$ is the largest of the two primes, then $d+2+p+q\leq 2k+1$ and we also deduce that $d\leq 2\left\lfloor\frac{k}{20}\right\rfloor-1$.

Denote by $\epsilon$ the number of boxes not on the first or second row. Let $r$ be such that the first row consists of $k+r-\eps+1$ boxes. It follows that the second row consists of $k-r$ boxes, and so $\epsilon\leq 2r+1$. Since $k-r\geq 1+c= p-d\geq k-3\lfloor\frac{k}{20}\rfloor+1$, we have that \[r\leq 3\left\lfloor\frac{k}{20}\right\rfloor-1\leq \frac{3}{20}k, \quad\textup{ and } \quad \epsilon\leq 6\left\lfloor\frac{k}{20}\right\rfloor-1.\] 

We use Lemma \ref{TwoPartitions2} to bound the hook product of this partition 

\begin{align*}\Pi &\leq \frac{(k+r-\eps+2)!(k-r)!}{2r-\eps+1}\eps!\frac{k+r-\eps+3}{k+r-2\eps+3}\frac{k-r+1}{k-r+1-\eps}\\
&\leq 4\frac{(k+r)!(k-r)!}{(2r-\eps+1)(k+r-\eps+2)^{\eps-2}}!\eps!,
\end{align*}
since the two last factors are easily seen to be bounded by 2 by using the upper bounds on $\eps$ and $r$. We use Lemma~\ref{Factoriallemma} to further bound the hook product

\begin{align*}\Pi &\leq (2k+1)k!^2\frac{2e^2}{\pi}e^{\frac{k+r}{k-r}\frac{r^2}{k}}\left(\frac{1}{k+r-\epsilon}\right)^{\epsilon-1} \eps!.
\end{align*}

Since by assumption $\Pi$ is either $\frac{1}{2}\Pi_k$, $\Pi_k$ or $2\Pi_k$, we deduce that \begin{equation*}\frac{1}{2}\leq \frac{2e^2}{\pi}e^{\frac{k+r}{k-r}\frac{r^2}{k}}\left(\frac{1}{k+r-\epsilon}\right)^{\epsilon-1} \eps!.
\end{equation*}

This means that our situation satisfies the conditions of Lemma \ref{ComputationCase2and3}, with $\eta=1$ and we may conclude that either $\eps=0,1$ or $r-2\eps\geq \sqrt{k}/2$, and $5\eps\leq r-2$. We deal with the case $\eps\neq 0$ first. We now consider the prime factors of the hook product. We see that the hook product contains as factors $\frac{(k+r-2\epsilon)!}{2r-\epsilon+2}$, and $(k-r-\epsilon)!$. 
Since $r-2\epsilon\leq r\leq \frac{3}{20}k$, and 

\[r-2\epsilon\geq \frac{1}{2}\sqrt{k},\] we can consider, according to Lemma~\ref{Primelemma2} a prime $p'>3(r-2\epsilon)$ such that $p' \mid \frac{(k+r-2\epsilon)!}{k!}$. Now since 
\begin{align*}
5\epsilon \leq r-2,
\end{align*}
we deduce that $p'>2r-\epsilon+2$ and so $p' \nmid \frac{k!}{(k-r-\epsilon)!}$, and arrive at a contradiction; there are more factors of $p'$ in $\Pi$ then there are in $\Pi_k$. We conclude that no partition satisfying the conditions of Case 3 with $\epsilon\neq 0,1$ has a hook product equal to $\frac{1}{2}\Pi_k, \Pi_k$, or $2\Pi_k$.

Finally, the cases $\epsilon=0,1$. If $\eps=0$, then the hook product of the partition is too big; \[\frac{(k+r+2)!(k-r)!}{2r+2} = \frac{k+r+2}{2r+2}(k+r+1)(k+r)!(k-r)!>2\Pi_k.\]

If $\eps=1$, then the hook product equals \[\frac{(k+r+2)!(k-r+1)!}{(2r+1)(k+r+1)(k-r)}=\frac{k+r+2}{2r+1}\frac{k-r+1}{k-r}(k-r)!(k+r)!.\]

Since for $r=0$, the hook product does not equal $\frac{1}{2}\Pi_k, \Pi_k$, or $2\Pi_k$, the hook product is smaller than $\frac{k}{2}(k-r)!(k+r)!$. It would be strictly smaller than $\frac{1}{2}\Pi_k$ unless $(k-r)!(k+r)!\geq 2k!^2$, thus, using Lemma \ref{Factoriallemma} \[\frac{e^2}{2\pi}e^{\frac{2r^2}{k}}\geq 2.\]

This immediately implies that we have $r\geq \sqrt{k}/2$, and we can proceed as in the case $\epsilon\neq 0$.

The fourth case, $d>k-3\left\lfloor\frac{k}{20}\right\rfloor-3\geq a$, corresponds to the partitions conjugate to those that we considered in case 3. Since conjugate partitions have the same hook product, no partition in this case has hook product equal to $\frac{1}{2}\Pi_k, \Pi_k$, or $2\Pi_k$, which finishes the proof.
\end{proof}

\section{Computer check}
In this section we describe the computer calculations we performed to check the remainder range of $n$ in Theorem~\ref{Theorem1}. We start with the odd case where $n=2k+1$, and we postpone the discussion regarding even $n$  to the end of this section. We have used the computer algebra program SAGE, because of its great functionality for integer partitions. The SAGE worksheet we used is available on the author's website. We will say that a partition $\mu=(\mu_1,\cdots,\mu_s)$ is contained in a partition $\nu=(\nu_1,\dots,\nu_r)$ if $s\leq r$ and $\mu_i\leq\nu_i$ for all $i=1,\cdots,s$. The following fact completes the proof of Theorem~\ref{Theorem1} for all odd $n\geq 5$. Recall that $\lambda_k =(k+1,1^k)$ and $\Pi_k=(2k+1)k!^2$. 
 
\begin{fact}\label{Fact1} The only partition of $2k+1$ with hook product equal to $\frac{1}{2}\Pi_k$, $\Pi_k$, or $2\Pi_k$, is $\lambda_k$, for all $2\leq k \leq 337$. 
\end{fact}

The first step is to use a naive algorithm for all $2\leq k\leq 34$, that is, simply running over all partitions of $2k+1$ and computing the hook product; this can be done in 30 minutes.

We define our auxiliary primes as follows. We let $q<p\leq k$ be the two biggest primes in $[1,k]$, and let $r$ be the biggest prime in $[1,\lfloor k/2 \rfloor]$. We note the following preliminary claim, which implies all inequalities we use further on.
\begin{claim}\label{prelimclaim} For all $k\in[35,337]$, and $q,p,r$ as defined above, we have that \[\begin{cases}3q>2k+1\\ 2q+\frac{p-1}{2}+2>2k+1\end{cases}.\]
When $k\neq 40,57$, we furthermore have that \[\begin{cases}2q+r>2k+1\\ q+3r>2k+1\\3r<2q\\5r>2k+1\end{cases}.\]
\end{claim}
\begin{verification} This is a straightforward check by computer.
\end{verification}
The strategy is now to use the fact that $p^2q^2r^4\mid \Pi_k$, to restrict the possibilities for the shape of the partitions that have $\frac{1}{2}\Pi_k$, $\Pi_k$, or $2\Pi_k$ as a hook product.

\begin{claim}\label{Claim2p} For all $k\in[35,337]$, a partition of $2k+1$ with hook product equal to $\frac{1}{2}\Pi_k$, $\Pi_k$, or $2\Pi_k$, cannot have a box of hook length $2p$ or $2q$.
\end{claim}
\begin{verification}
We first show that a hypothetical box of hook length $2p$ or $2q$ should lie on position $(1,1)$. 
If that is not the case, Lemma \ref{factorialdivides} gives us that $(4q-2k-1)!$ divides the hook product. Note that by definition $4q-2k-1<2k+1$. However, one may verify that there is prime in $[k+1,4q-2k-1]$ for all $35\leq k \leq 337$, which then divides the hook product of this partition but does not divide $\frac{1}{2}\Pi_k$, $\Pi_k$, or $2\Pi_k$. So a box of hook length $2p$ or $2q$ can only be present on position $(1,1)$.

Now we eliminate the possibility that the box on position $(1,1)$ has hook length $2p$. We use a computer to run over all partitions of $2k+1$ having a box of hook length $2p$ on position $(1,1)$, and having a box of hook length $q$ in both the first row and column -- which they necessarily have since $2p+q>2k+1$. Denote by $a$ the number of boxes in the leg of the box of hook length $q$ on position $(1,c+1)$, and let $b$ be the number of boxes in the arm of the box of hook length $q$ on position $(d+1,1)$. Since this is a partition, $b\geq c \Longleftrightarrow a\geq d$. Because conjugate partitions have the same hook product, we may assume $c\geq d$. Furthermore it holds that \[2p=2q-1+d+c-a-b,\] and since the total number of boxes should be at most $2k+1$, \begin{equation}\label{eq:range1}
2p+bd+\max(a(c-b),0) + \max(c(a-d),0)\leq 2k+1.
\end{equation}

We will check for all possible tuples $(a,b,c)$ -- which fixes the value of $d$ -- that the hook product of all corresponding partitions does not equal $\frac{1}{2}\Pi_k, \Pi_k$ or $2\Pi_k$. We bound the range of the variables using the inequalities $a+b+2p\leq 2k+1$ and $c+d-1+2q-1\leq 2k+1$. We run over the tuples $(a,b,c)$ in the following range \[\begin{cases}a\in[0,2k+1-2p]\\b\in[0,2k+1-2p-a]\\c\in[1,2k+1-2q+1]\\d=2p-2q+1-c+a+b\end{cases},\] and in each loop we first check that \eqref{eq:range1} holds, that $1\leq d\leq c$ and that $b\geq c \Longleftrightarrow a\geq d$ holds. Then we compute the hook products of all partitions of this form. To make sure we only consider partitions corresponding to the described situation, we distinguish two cases. If $b<c$, we check all partitions containing \[(c+q-a, (c+1)^a, (b+1)^{d-a}, 1^{q-b-1}),\] and which are contained in \[((c+q-a)^{a+1}, c^{d-a-1}, (b+1)^{q-b}).\] If $b\geq c$, we check all partitions containing \[(c+q-a, (b+1)^d, (c+1)^{a-d}, 1^{q-b-1-(a-d)}),\] and which are contained in \[((c+q-a)^{d},(b+1)^{a-d+1}, c^{q-b-1-(a-d)}).\]

However, this check is impractically slow to implement, in particular for those $k$ with small $q$. Therefore, we consider also the position of the four boxes that have hook length a multiple of $r$, when $k\neq 40,57$. Since $2p+r>2k+1$, no box of hook length at least $r$ can occur outside the first row or column. Since $3r+q>2k+1$, no box of hook length at least $3r$ can occur on the first row or column, and we deduce that a box of hook length $2r$ (and $r$) must be present in both the first row and column. We may thus, for every $k\neq40,57$, replace $q$ by $2r$ to significantly speed up the algorithm. 

Finally we check for the case of a box with hook length $2q$ on position $(1,1)$, and boxes of hook length $p$ on both the first row and column, completely analogously as above with $p$ and $q$ interchanged. Again, since $2q+r>2k+1$ and $3r+p>2k+1$, we may use $2r$ instead of $p$ for $k\neq 40,57$.

This check has been successfully performed for all integers $k \in [35,337]$ in one hour.         
\end{verification}

We now reduce the possible shape of a partition of $2k+1$ with hook product $\frac{1}{2}\Pi_k, \Pi_k, $ or $2\Pi_k$ to three qualitative cases.

\begin{claim}\label{Claimshape} For all $k\in[35,337]$, a partition of $2k+1$ with hook product $\frac{1}{2}\Pi_k, \Pi_k, $ or $2\Pi_k$ either has boxes of hook length $p$ and $q$ on both the first row and column, on both the first and second row, or on both the first and second column. \end{claim}
\begin{verification} Since $3q>2k+1$, there cannot be any boxes of hook length $3q$ or larger. By Claim~\ref{Claim2p}, boxes of hook length $2p$ or $2q$ are also impossible, so that we have two boxes of hook length $p$ and two boxes of hook length $q$. Denote the arms and legs of the two boxes of hook length $p$ by $\alpha, \beta,\gamma,\delta$, and let $a,b,c,d$ be the number of boxes in $\alpha, \beta,\gamma,\delta$, as in the proof of Theorem~\ref{Theorem1.1}. Each box of hook length $p$ has one leg containing at least $\frac{p-1}{2}$ boxes. There are then four cases to be distinguished: \\
\textbf{ Case 1:} $a, c \geq \frac{p-1}{2}$. Since $2p-1 +  \frac{p-1}{2} >2k+1$, this means that $\alpha$ and $\gamma$  and are on the first and second row.
If there would be a box of hook length $q$ not on the first two rows, then the number of boxes would be at least $p+q+(p-1)/2>2k+1$, a contradiction.\\
\textbf{ Case 2:} $a, d \geq \frac{p-1}{2}$. As before, this implies that $\alpha$ and $\delta$ are on the first row and column. If there would a box of hook length $q$ not on the first row or column, then the number of boxes would be at least $2q+(p-1)/2+2>2k+1$, contradiction.\\
\textbf{ Case 3:} $b, c \geq \frac{p-1}{2}$. As before, this implies that $\beta$ and $\gamma$ are on the second row and column, so that the boxes of hook length $p$ are in positions $(1,2)$ and $(1,2)$. It follows that $a\geq c-1\geq \frac{p-3}{2}$, and likewise $d\geq \frac{p-3}{2}$. There can potentially be a box of hook length $q$ on position $(2,2)$, but the other box of hook length $q$ has one coordinate at least 3. Thus the number of boxes is at least $p+q+\frac{p-1}{2}>2k+1$, which is a contradiction, and so this case cannot occur.\\
\textbf{ Case 4:} $b, d \geq \frac{p-1}{2}$. This occurs exactly for the conjugates of the partitions in case 1. This then corresponds to the case where there is a box of hook length $p$ and $q$ are on both the first and second column.
\end{verification}

We thus have 3 possible configurations of the boxes of hook length $p$ and $q$. In the following two claims, we check that in each of the 3 cases, such partitions do not have hook product equal to $\frac{1}{2}\Pi_k, \Pi_k, $ or $2\Pi_k$.
\begin{claim}\label{ClaimHor} For all $k\in[35,337]$, a partition of $2k+1$ with hook product $\frac{1}{2}\Pi_k, \Pi_k, $ or $2\Pi_k$ cannot have boxes of hook length $p$ and $q$ on both the first and second row.
\end{claim}
\begin{verification}
We introduce some notation. Let the boxes of hook length $p$ be on positions $(2,x)$ and $(1,y)$, with legs numbering $a$ and $c$ boxes respectively, and let the boxes of hook length $q$ be on positions $(2,z)$ and $(1,u)$, with legs containing $b$ and $d$ boxes respectively. The number of boxes on the first and second row will be denoted by $R_1$ and $R_2$. 

We first state some immediate observations using the definition of the hook length and the fact that this is a partition, \[\begin{cases}x+p-a-1=R_2\\z+q-b-1=R_2\\y+p-c-1=R_1\\u+q-d-1=R_1\end{cases},\begin{cases}x<y\\x<z\\y<u\\z<u\end{cases} ,\begin{cases}a\geq b\\ c\geq d\end{cases}.\] 

It follows that \[\begin{cases} z\leq x+(p-q)\\ u\leq y+(p-q)\end{cases}.\]

We will check for all possible tuples $(x,y,z,u,R_1,R_2)$ -- which fixes the values of $a,b,c,d$ -- that the hook product of all corresponding partitions does not equal $\frac{1}{2}\Pi_k, \Pi_k$ or $2\Pi_k$.

We deduce bounds for the length of the two rows. Since $d-1+2p\leq 2k+1$, we have that \[R_1\geq u+q-(2k+1)-2p-2.\]

If $z\neq y$, we have that $b+2p\leq 2k+1$, and also $z\geq 2$, so that \[z\neq y \Longrightarrow R_2\geq q+2p-2k.\]

We also have that $d\geq 0$ and so $R_1\leq u+q-1.$ 
Also note that $2x-1+2p-1\leq 2k+1$ and so \[x\leq k-p+1.\]
 
We now distinguish 3 different cases according to the relative position of the box of hook length $p$ on the first row and the box of hook length $q$ on the second row.

\noindent \textbf{ Case 1:} $x<z<y<u$, and so $a\geq b\geq c-1 \geq d-1$.

From $b\geq c-1$ it follows that $R_1-R_2 \geq y-z+(p-q)-1$. Note that $y-1+x-1\leq 2k+1-(2p-1)$. We can now run over the tuples $(x,y,z,u,R_1,R_2)$ in the following range

\[\begin{cases}x\in[1,k-p+1]\\z\in[x+1,x+(p-q)]\\y\in[z+1,2k-2p+4-x]\\u\in[y+1,y+(p-q)]\\R_1\in[u+q-(2k+1)-2p-2,u+q-1]\\R_2\in[q+2p-2k,R_1-y+z-(p-q)+1]\end{cases}.\]

In each loop we first check if there can be any partitions with this particular tuple by checking that the number of boxes is at most $2k+1$; that is we check whether $R_1+R_2 +(d-1)u+(c-d)y+(b-c+1)z+(a-b)x\leq 2k+1$. Then we compute the hook products of all partitions of this form; more precisely of all partitions containing \[(R_1,R_2,u^{d-1}, y^{c-d},z^{b-c+1},x^{a-b}),\] which are contained in the partition\[(R_1,R_2^{d},(u-1)^{c-d},(y-1)^{b-c+1},(z-1)^{a-b},(x-1)^{2k+1}).\]

\noindent \textbf{ Case 2:} $x<y=z<u$, and so $a\geq b= c-1 \geq d-1$.

From $b= c-1$ it follows that $R_1-R_2 = y-z+(p-q)-1$. We can now run over the tuples $(x,y,z,u,R_1,R_2)$ in the following range

\[\begin{cases}x\in[1,k-p+1]\\y\in[x+1,x+(p-q)]\\u\in[y+1,y+(p-q)]\\R_1\in[u+q-(2k+1)-2p-2,u+q-1]\\R_2=R_1+1-(p-q)\end{cases}.\]

In each loop we first check if there can be any partitions with this particular tuple by checking that the number of boxes is at most $2k+1$; that is we check whether $R_1+R_2+(d-1)u+(c-d)y+(a-c+1)x \leq 2k+1$. Then we compute the hook products of all partitions of this form; more precisely of all partitions containing \[(R_1,R_2,u^{d-1}, y^{c-d},x^{a-c+1}),\] which are contained in the partition\[(R_1,R_2^{d},(u-1)^{c-d},(y-1)^{a-c+1},(x-1)^{2k+1}).\]

\noindent \textbf{ Case 3:}  $x<y<z<u$, and so $a\geq c-1\geq b \geq d-1$.

From $a\geq c-1\geq b \geq d-1$ it follows that \[\begin{cases}y-x-1\\u-z-1\end{cases} \leq R_1-R_2 \leq y-z+(p-q)-1.\] 

We can now run over the tuples $(x,y,z,u,R_1,R_2)$ in the following ranges 
\[\begin{cases}x\in[1,k-p+1]\\y\in[x+1,x+(p-q)-1]\\z\in[y+1,x+(p-q)]\\u\in[z+1,y+(p-q)]\\R_1\in[u+q-(2k+1)-2p-2,u+q-1]\\R_2\in[\max(q+2p-2k,R_1-y+z-(p-q)+1),\min(R_1-y+x+1,R_1-u+z+1)]\end{cases}.\]

In each loop we first check if there can be any partitions with this particular tuple by checking that the number of boxes is at most $2k+1$; that is we check whether $R_1+R_2+(d-1)u+(b-d+1)z+(c-b-1)y+(a-c+1)x \leq 2k+1$. Then we compute the hook products of all partitions of this form; more precisely of all partitions containing \[(R_1,R_2,u^{d-1}, z^{b-d+1},y^{c-b-1},x^{a-c+1}),\] which are contained in the partition\[(R_1,R_2^{d},(u-1)^{b-d+1},(z-1)^{c-b-1},(y-1)^{a-c+1},(x-1)^{2k+1}).\]

As in Claim \ref{Claim2p}, we consider also the position of the four boxes that have hook length a multiple of $r$, when $k\neq 40,57$. Since $p+q+r>2k+1$, no box of hook length at least $r$ can occur outside the first two rows. Since $3r+q>2k+1$, no box of hook length at least $3r$ can occur, and we deduce that a box of hook length $2r$ (and $r$) must be present in both the first two rows. We may thus, for every $k\neq40,57$, replace $(p,q)$ by $(p,2r)$ or $(2r,q)$, depending on whether $p>2r>q$ or $2r>p$, to significantly speed up the algorithm.

This check has been successfully performed for all integers $k \in [35,337]$ in about 5 minutes.
\end{verification}

Clearly, this also shows that a partition with hook product $\frac{1}{2}\Pi_k, \Pi_k, $ or $2\Pi_k$ cannot have boxes of hook length $p$ and $q$ on both the first and second column. The next claim deals with the remaining option.

\begin{claim}\label{ClaimHorVer} For all $k\in[35,337]$, a partition of $2k+1$ different from $\lambda_k$ with hook product $\frac{1}{2}\Pi_k, \Pi_k,$ or $2\Pi_k$ cannot have boxes of hook length $p$ and $q$ on both the first row and column.\end{claim}
\begin{verification}
We introduce some notation. Let the boxes of hook length $p$ be on positions $(1,x+1)$ and $(z+1,1)$, with leg and arm respectively denoted by $\alpha$ and $\gamma$. Let the boxes of hook length $q$ be on positions $(1,y+1)$ and $(u+1,1)$, with leg and arm respectively denoted by $\beta$ and $\delta$. We will denote the cardinality of $\alpha, \beta, \gamma, \delta$ by $a,b,c,d$. The number of boxes on the first row and column will be denoted respectively by $R_1$ and $C_1$. 

We first state some immediate observations using the definition of the hook length and the fact that this is a partition, \[\begin{cases}x+p-a=R_1\\y+q-b=R_1\\z+p-c=C_1\\u+q-d=C_1\end{cases},\begin{cases}x<y\\d\leq c\end{cases} ,\begin{cases}a\geq b \\z< u\end{cases}.\] 

It follows that \[\begin{cases}  y\leq x+(p-q)\\ c< d+(p-q)\end{cases}.\]

We will check for all possible tuples $(x,y,c,d,C_1,R_1)$ -- which fixes the values of $a,b,z,u$ -- that the hook product of all corresponding partitions does not equal $\frac{1}{2}\Pi_k, \Pi_k$ or $2\Pi_k$, unless the partition is $\lambda_k$.

Note that since conjugated partitions have the same hook product, we may assume that $z\geq x$. From this we deduce that \[C_1\geq x+p-c,\] and also since $x+z-1+2p-1\leq 2k+1$, that \[x\leq k-p+1.\] 

Since $b\geq 0$, we have that \[R_1\leq y+q.\] 

Now note that $z\min(c+1,x)+2p-1\leq 2k+1$, from which it follows that \[C_1\leq p-c+\frac{2k-2p+2}{\min(c+1,x)}.\]

We distinguish 5 cases according to how many intersections there are between the arms and legs $\alpha$, $\beta$, $\gamma$, $\delta$.

\noindent \textbf{Case 0: 0 intersections} : $y>x>c\geq d$, and so $b\leq a<z<u$.

From $a<z$ it follows that $2p+x-c<C_1+R_1$. We can now run over the tuples $(x,y,c,d,C_1,R_1)$ in the following range

\[\begin{cases}d\in[0,k-p]\\c\in[d,d+(p-q)-1]\\x\in[c+1,k-p+1]\\y\in[x+1,x+(p-q)]\\C_1\in[x+p-c,p-c+\lfloor\frac{2k-2p+2}{c+1}\rfloor]\\R_1\in[2p-c+x-C_1+1,y+q]\end{cases}.\]

In each loop we first check if there can be any partitions with this particular tuple by checking that the number of boxes is at most $2k+1$; that is we check whether $R_1+C_1-1+by+(a-b)x+(z-a)c+(u-z)d \leq 2k+1$. Then we compute the hook products of all partitions of this form; more precisely of all partitions containing \[(R_1,(y+1)^{b}, (x+1)^{a-b},(c+1)^{z-a},(d+1)^{u-z},1^{C_1-u-1}),\] which are contained in the partition\[(R_1^{b+1},y^{a-b},x^{z-a-1},(c+1)^{u-z},(d+1)^{C_1-u}).\]

\noindent \textbf{Case 1 : 1 intersection} : $y>c\geq x> d$, and so $b<z\leq a<u$.

From $b<z\leq a<u$ it follows that \[\begin{cases}p+q+y-c\\p+q+x-d\end{cases}<C_1+R_1\leq 2p+x-c.\] 
We can now run over the tuples $(x,y,c,d,C_1,R_1)$ in the following range

\[\begin{cases}d\in[0,k-p]\\x\in[d+1,k-p+1]\\c\in[x,d+(p-q)-1]\\y\in[c+1,x+(p-q)]\\C_1\in[x+p-c,p-c+\lfloor\frac{2k-2p+2}{x}\rfloor]\\R_1\in[\max(p+q+y-c-C_1,p+q+x-d-C_1)+1,\min(2p-c+x-C_1,y+q)]\end{cases}.\]

In each loop we first check if there can be any partitions with this particular tuple by checking that the number of boxes is at most $2k+1$; that is we check whether $R_1+C_1-1+by+(z-b)c+(a-z)x+(u-a)d\leq 2k+1$. Then we compute the hook products of all partitions of this form; more precisely of all partitions containing \[(R_1,(y+1)^{b}, (c+1)^{z-b},(x+1)^{a-z},(d+1)^{u-z},1^{C_1-u-1}),\] which are contained in the partition\[(R_1^{b+1},y^{z-b-1},(c+1)^{a-z+1},x^{u-a-1},(d+1)^{C_1-u}).\]

\noindent \textbf{Case 2 : 2 intersections} : We show that this case is actually impossible.

Up to conjugation, we are in the situation that $c\geq y>x>d$, and so $z\leq b\leq a<u$. From this it follows that $p-q = (y-x)+(a-b)<(c-d)+(u-z)=p-q$, a contradiction.

\noindent \textbf{Case 3 : 3 intersections} : $c\geq y> d\geq x$, and so $z\leq b<u\leq a$.

From $z\leq b<u\leq a$ it follows that \[2q+y-d<C_1+R_1\leq \begin{cases}p+q+y-c\\p+q+x-d\end{cases}.\] 

We can now run over the tuples $(x,y,c,d,C_1,R_1)$ in the following range 

\[\begin{cases}x\in[1,k-p+1]\\d\in[x,x+(p-q)-1]\\y\in[d+1,x+(p-q)]\\c\in[y,d+(p-q)-1]\\C_1\in[x+p-c,p-c+\lfloor\frac{2k-2p+2}{x}\rfloor]\\R_1\in[2q-d+y-C_1+1,\min(p+q+y-c-C_1,p+q+x-d-C_1)]\end{cases}.\]

In each loop we first check if there can be any partitions with this particular tuple by checking that the number of boxes is at most $2k+1$; that is we check whether $R_1+C_1-1+zc+(b-z)y+(u-b)d+(a-u)x \leq 2k+1$. Then we compute the hook products of all partitions of this form; more precisely of all partitions containing \[(R_1,(c+1)^{z}, (y+1)^{b-z},(d+1)^{u-b},(x+1)^{a-u},1^{C_1-a-1}),\] which are contained in the partition \[(R_1^{z},(c+1)^{b-z+2},y^{u-b-1},(d+1)^{a-u+1},x^{C_1-a-1}).\]

\noindent \textbf{Case 4 : 4 intersections} : $c\geq d\geq y>x$, and so $z<u\leq b\leq a$.

From $u\leq b$ it follows that $C_1+R_1\leq 2q+y-d$. Note that $xa\leq 2k+1-2q$, and so \[R_1\geq x-p-\lfloor\frac{2k-2q+1}{x}\rfloor.\] 

We also have that $d-1\leq 2k+1-2p$, and we can now run over the tuples $(x,y,c,d,C_1,R_1)$ in the following range

\[\begin{cases}x\in[1,k-p+1]\\y\in[x+1,x+(p-q)]\\d\in[y,2k-2p+2]\\c\in[d,d+(p-q)-1]\\C_1\in[x+p-c,p-c+\lfloor\frac{2k-2p+2}{x}\rfloor]\\R_1\in[x+p-\lfloor\frac{2k-2q+1}{x}\rfloor,2q-d+y-C_1]\end{cases}.\]

In each loop we first check if there can be any partitions with this particular tuple by checking that the number of boxes is at most $2k+1$; that is we check whether $R_1+C_1-1+zc+(u-z)d+(b-u)y+(a-b)x\leq 2k+1$. Then we compute the hook products of all partitions of this form; more precisely of all partitions containing \[(R_1,(c+1)^{z}, (d+1)^{u-z},(y+1)^{b-u},(x+1)^{a-b},1^{C_1-a-1}),\] which are contained in the partition \[(R_1^{z},(c+1)^{u-z},(d+1)^{b-u+1},y^{a-b},x^{C_1-a-1}).\]

As in Claim \ref{Claim2p}, we consider also the position of the four boxes that have hook length a multiple of $r$, when $k\neq 40,57$. Since $2q+r>2k+1$, no box of hook length at least $r$ can occur outside the first  row and column. Since $3r+q>2k+1$, no box of hook length at least $3r$ can occur except on position $(1,1)$. Assuming that the hook length of the box $(1,1)$ is not a multiple of $r$, we may deduce that a box of hook length $2r$ (and $r$) must be present in both the first row and column. We may thus, for every $k\neq40,57$, replace $(p,q)$ by $(p,2r)$ or $(2r,q)$, depending on whether $p>2r>q$ or $2r>p$, to significantly speed up the algorithm, provided we check also the case that the box on position $(1,1)$ is $3r$ or $4r$, since $5r>2k+1$.

In light of the bounds on $R_1+C_1$ in each case, we see that $3r$ can only occur in Case 4, since $3r<2q$. A box of hook length $4r$ cannot occur in Case 4 since $4r>2q$(for relevant $r$), but could occur in Case 1 or 3 if $q<2r<p$ or in Case 0 if $p<2r$.

This check has been successfully performed for all integers $k \in [35,337]$ in little over one hour. \end{verification}

This completes the description of the computer verification of Fact~\ref{Fact1}. 
To complete the proof of Theorem~\ref{Theorem1}, we describe how to check the following fact by computer. Recall that $\lambda^{'}_k=(k+1,2,1^{k-1})$, and $\Pi^{'}_k=(2k+1)(k+1)^2(k-1)!^2$.
\begin{fact}\label{Fact2} The only partition of $2k+2$ with hook product equal to $\frac{1}{2}\Pi^{'}_k$, $\Pi^{'}_k$, $2\Pi^{'}_k$, is $\lambda^{'}_k$, for all $2\leq k \leq 337$. \end{fact}

Only minor changes to the above strategy (and code) are required to check the even case. One needs to define the auxiliary primes $p,q,r$ slightly differently; $q<p$ are the two biggest primes in $[1,k-1]\cup \{k+1\}$, and $r$ is the biggest prime in $[1,\lfloor\frac{k-1}{2}\rfloor]$, or $r=\frac{k+1}{2}$ if that is a prime. Thus it is again ensured that $p^2q^2r^4\mid\Pi^{'}_k$. Replacing all instances of $2k+1$ by $2k+2$, the proofs and verifications go through as above, except in the following points
\begin{enumerate}\item The analogue of the preliminary Claim~\ref{prelimclaim} does not go through for $k=37$ and $k=41$, which we therefore need to check by some other method. We use the naive algorithm for $k=37$, and for $k=41$ we use the fact that $83=2k+1\mid \Pi^{'}_k$, which is a prime, so that there exists a box of hook product $n-1$. A quick check then suffices for this case. 
\item The bounds on $r$ in the analogue of the preliminary Claim~\ref{prelimclaim} do not hold for $k=58$, and so for $k=58$ we cannot speed up our algorithms using the prime $r$ -- just as we couldn't use it for $k=40,57$ in the odd case.
\end{enumerate}

Except for the naive verification of $k=37$ which in itself takes half an hour, the even case takes about as much time as the odd case, and has been verified in about 3 hours.

\section*{Acknowledgements} 

We wish to thank Jan-Christoph Schlage-Puchta for bringing this problem to our attention. With regards to the presentation of the arguments and the readibility of the manuscript we are thankful for the many helpful comments of the anonymous referee.

\bibliographystyle{plain}
\bibliography{biblio}

\begin{thebibliography}{1}

\bibitem{BHP}
R.C. {Baker}, G.~{Harman}, and J.~{Pintz}.
\newblock {The difference between consecutive primes. II.}
\newblock {\em {Proc. Lond. Math. Soc. (3)}}, 83(3):532--562, 2001.

\bibitem{BBOO}
Antal Balog, Christine Bessenrodt, J{\o}rn~B. Olsson, and Ken Ono.
\newblock Prime power degree representations of the symmetric and alternating
  groups.
\newblock {\em J. London Math. Soc. (2)}, 64(2):344--356, 2001.

\bibitem{FultonHarris}
William Fulton and Joe Harris.
\newblock {\em Representation theory : a first course}.
\newblock Graduate texts in mathematics. Springer-Verlag, New York, Berlin,
  Paris, 1991.

\bibitem{Sharp}
Hongquan {Liu} and Jie {Wu}.
\newblock {Numbers with a large prime factor.}
\newblock {\em {Acta Arith.}}, 89(2):163--187, 1999.

\bibitem{Rama}
K.~{Ramachandra}.
\newblock {A note on numbers with a large prime factor.}
\newblock {\em {J. Lond. Math. Soc., II. Ser.}}, 1:303--306, 1969.

\bibitem{Schoenfeld}
Lowell Schoenfeld.
\newblock Corrigendum: ``{S}harper bounds for the {C}hebyshev functions
  {$\theta (x)$} and {$\psi (x)$}. {II}'' ({M}ath. {C}omput. {\bf 30} (1976),
  no. 134, 337--360).
\newblock {\em Math. Comp.}, 30(136):900, 1976.

\bibitem{Tong}
Hung~P. {Tong-Viet}.
\newblock {Symmetric groups are determined by their character degrees.}
\newblock {\em {J. Algebra}}, 334(1):275--284, 2011.

\bibitem{TongAlt}
Hung~P. {Tong-Viet}.
\newblock {Alternating and sporadic simple groups are determined by their
  character degrees.}
\newblock {\em {Algebr. Represent. Theory}}, 15(2):379--389, 2012.

\end{thebibliography}
{\sc Mathematisches Institut, Georg-August Universit\"at G\"ottingen, Germany.}
\small\tt Email address : {\bf kdebaen@mathematik.uni-goettingen.de}
\end{document}